\pdfoutput=1
\newif\ifpersonal
\RequirePackage[l2tabu,orthodox]{nag} 
\documentclass[10pt,a4paper,reqno]{amsart} 
\usepackage{amsmath,amsthm,amssymb,mathrsfs,mathtools,bm,eucal,tensor} 
\usepackage{microtype,fixltx2e,lmodern} 
\usepackage[utf8]{inputenc} 
\usepackage[T1]{fontenc} 
\usepackage{enumerate,comment,braket,xspace,tikz-cd,csquotes} 
\usepackage[all,cmtip]{xy} 
\usepackage[centering,vscale=0.7,hscale=0.8]{geometry}
\usepackage[hidelinks]{hyperref}
\usepackage[capitalize]{cleveref}

\theoremstyle{plain}
\newtheorem{thm-intro}{Theorem}
\newtheorem{thm}{Theorem}[section]
\newtheorem*{thm*}{Theorem}

\newtheorem{lem}[thm]{Lemma}
\newtheorem{prop}[thm]{Proposition}

\newtheorem{cor}[thm]{Corollary}

\theoremstyle{definition}
\newtheorem{defin}[thm]{Definition}

\newtheorem{eg}[thm]{Example}

\theoremstyle{remark}
\newtheorem{rem}[thm]{Remark}
\numberwithin{equation}{section}

\ifpersonal
\newcommand*{\personal}[1]{\textcolor[rgb]{0.6,0.6,1}{(Personal: #1)}}
\newcommand*{\todo}[1]{\textcolor{red}{(Todo: #1)}}
\else
\newcommand*{\personal}[1]{\ignorespaces}
\newcommand*{\todo}[1]{\ignorespaces}
\fi

\newcommand{\rB}{\mathrm B}

\newcommand{\rH}{\mathrm H}

\newcommand{\rR}{\mathrm R}

\newcommand{\cA}{\mathcal A}

\newcommand{\cC}{\mathcal C}
\newcommand{\cD}{\mathcal D}

\newcommand{\cF}{\mathcal F}
\newcommand{\cH}{\mathcal H}
\newcommand{\cG}{\mathcal G}

\newcommand{\cJ}{\mathcal J}

\newcommand{\cO}{\mathcal O}

\newcommand{\cS}{\mathcal S}
\newcommand{\cT}{\mathcal T}

\newcommand{\cV}{\mathcal V}

\newcommand{\cX}{\mathcal X}
\newcommand{\cY}{\mathcal Y}

\DeclareFontFamily{U}{BOONDOX-calo}{\skewchar\font=45 }
\DeclareFontShape{U}{BOONDOX-calo}{m}{n}{<-> s*[1.05] BOONDOX-r-calo}{}
\DeclareFontShape{U}{BOONDOX-calo}{b}{n}{<-> s*[1.05] BOONDOX-b-calo}{}
\DeclareMathAlphabet{\mathcalboondox}{U}{BOONDOX-calo}{m}{n}

\makeatletter
\let\save@mathaccent\mathaccent
\newcommand*\if@single[3]{%
	\setbox0\hbox{${\mathaccent"0362{#1}}^H$}%
	\setbox2\hbox{${\mathaccent"0362{\kern0pt#1}}^H$}%
	\ifdim\ht0=\ht2 #3\else #2\fi
}
\newcommand*\rel@kern[1]{\kern#1\dimexpr\macc@kerna}
\newcommand*\widebar[1]{\@ifnextchar^{{\wide@bar{#1}{0}}}{\wide@bar{#1}{1}}}
\newcommand*\wide@bar[2]{\if@single{#1}{\wide@bar@{#1}{#2}{1}}{\wide@bar@{#1}{#2}{2}}}
\newcommand*\wide@bar@[3]{%
	\begingroup
	\def\mathaccent##1##2{%
		\let\mathaccent\save@mathaccent
		\if#32 \let\macc@nucleus\first@char \fi
		\setbox\z@\hbox{$\macc@style{\macc@nucleus}_{}$}%
		\setbox\tw@\hbox{$\macc@style{\macc@nucleus}{}_{}$}%
		\dimen@\wd\tw@
		\advance\dimen@-\wd\z@
		\divide\dimen@ 3
		\@tempdima\wd\tw@
		\advance\@tempdima-\scriptspace
		\divide\@tempdima 10
		\advance\dimen@-\@tempdima
		\ifdim\dimen@>\z@ \dimen@0pt\fi
		\rel@kern{0.6}\kern-\dimen@
		\if#31
		\overline{\rel@kern{-0.6}\kern\dimen@\macc@nucleus\rel@kern{0.4}\kern\dimen@}%
		\advance\dimen@0.4\dimexpr\macc@kerna
		\let\final@kern#2%
		\ifdim\dimen@<\z@ \let\final@kern1\fi
		\if\final@kern1 \kern-\dimen@\fi
		\else
		\overline{\rel@kern{-0.6}\kern\dimen@#1}%
		\fi
	}%
	\macc@depth\@ne
	\let\math@bgroup\@empty \let\math@egroup\macc@set@skewchar
	\mathsurround\z@ \frozen@everymath{\mathgroup\macc@group\relax}%
	\macc@set@skewchar\relax
	\let\mathaccentV\macc@nested@a
	\if#31
	\macc@nested@a\relax111{#1}%
	\else
	\def\gobble@till@marker##1\endmarker{}%
	\futurelet\first@char\gobble@till@marker#1\endmarker
	\ifcat\noexpand\first@char A\else
	\def\first@char{}%
	\fi
	\macc@nested@a\relax111{\first@char}%
	\fi
	\endgroup
}
\makeatother

\newcommand{\PSh}{\mathrm{PSh}}
\newcommand{\Sh}{\mathrm{Sh}}

\newcommand{\rSet}{\mathrm{Set}}

\newcommand{\tauan}{\tau_\mathrm{an}}

\newcommand{\Mod}{\textrm{-}\mathrm{Mod}}
\newcommand{\Coh}{\mathrm{Coh}}

\newcommand{\Cohh}{\mathrm{Coh}^\heartsuit}

\newcommand{\Stn}{\mathrm{Stn}_{\mathbb C}}

\newcommand{\An}{\mathrm{An}}

\newcommand{\bfMap}{\mathbf{Map}}

\newcommand{\cTan}{\cT_{\mathrm{an}}}

\newcommand{\trunc}{\mathrm{t}_0}

\newcommand{\cHom}{\cH \mathrm{om}}

\newcommand{\St}{\mathbf{St}}

\newcommand{\Lan}{\mathrm{Lan}}

\newcommand{\Perf}{\mathrm{Perf}}

\newcommand{\Cat}{\mathrm{Cat}}

\newcommand{\bfPerf}{\mathbf{Perf}}

\newcommand{\Catst}{\Cat_\infty^{\mathrm{Ex}}}

\newcommand{\dAn}{\mathrm{dAn}}
\newcommand{\dStn}{\mathrm{dStn}_{\mathbb C}}
\newcommand{\Stnred}{\Stn^{\mathrm{red}}}
\newcommand{\dAnSt}{\mathrm{dAnSt}_{\mathbb C}}
\newcommand{\red}{_\mathrm{red}}
\newcommand{\DR}{_\mathrm{dR}}
\newcommand{\B}{_\mathrm{B}}
\newcommand{\topl}{^\mathrm{top}}
\newcommand{\AnRing}{\mathrm{AnRing}}
\newcommand{\bfAnPerf}{\mathbf{AnPerf}}
\newcommand{\bfAnCoh}{\mathbf{AnCoh}}

\newcommand{\an}{^\mathrm{an}}
\newcommand{\alg}{^\mathrm{alg}}

\newcommand{\ev}{\mathrm{ev}}

\newcommand{\inv}{^{-1}}

\newcommand{\canal}{$\mathbb C$-analytic\xspace}

\newcommand{\op}{^\mathrm{op}}

\usetikzlibrary{decorations.markings} 
\tikzset{
  closed/.style = {decoration = {markings, mark = at position 0.5 with { \node[transform shape, xscale = .8, yscale=.4] {/}; } }, postaction = {decorate} },
  open/.style = {decoration = {markings, mark = at position 0.5 with { \node[transform shape, scale = .7] {$\circ$}; } }, postaction = {decorate} }
}

\DeclareMathOperator{\Fun}{Fun}

\DeclareMathOperator{\Hom}{Hom}

\DeclareMathOperator{\Map}{Map}

\DeclareMathOperator{\Sp}{Sp}

\DeclareMathOperator*{\colim}{colim}

\begin{document}

\title{The derived Riemann-Hilbert correspondence}

\author{Mauro PORTA}
\address{Mauro PORTA, University of Pennsylvania, David Rittenhouse Laboratory, 209 South 33rd Street, Philadelphia, PA 19104, United States}
\email{maurop@math.upenn.edu}

\date{\today}

\begin{abstract}
	In this short paper we prove a derived version of the Riemann-Hilbert correspondence of Deligne and Simpson.
	Our generalization is twofold: on one side we consider families of representations of the full homotopy type of a smooth analytic space $X$ in the stable $\infty$-category of complexes of coherent sheaves and of perfect complexes; on the other side, we allow the base parametrizing such families to be derived \canal spaces.
\end{abstract}

\maketitle

\tableofcontents

\section{Introduction}

The Riemann-Hilbert correspondence is a classical problem, notorious for its richness and depth.
Over the years, it has received the attention of several illustrious mathematicians, such as P.\ Deligne \cite{Deligne_Equations_differentielles}, C.\ Simpson \cite{Simpson_Moduli_representations_I,Simpson_Moduli_representations_II}, M.\ Kashiwara \cite{Kashiwara_Riemann_Hilbert}.
Classically, the Riemann-Hilbert correspondence is an equivalence between the triangulated category of regular holonomic $\cD$-modules and that of constructible sheaves.
More recently, M.\ Kashiwara and A.\ D'Agnolo \cite{Kashiwara_DAgnolo_Irregular_riemann} extended this equivalence to all holonomic $\cD$-modules (see also the survey \cite{Kashiwara_Schapira_Regular_irregular}).

If instead of looking at all the constructible sheaves we restrict our attention to the full subcategory of locally constant sheaves, we obtain a particularly simple version of the Riemann-Hilbert correspondence.
On the $\cD$-module side, locally constant sheaves correspond to those $\cD$-modules whose underlying $\cO_X$-module is locally free.
With little extra work, this correspondence can be made in families.
We recall the precise statement:

\begin{thm-intro}[{P.\ Deligne, see \cite[Théorème 2.23]{Deligne_Equations_differentielles}}] \label{thm:Deligne}
	Let $p \colon X \to S$ be a smooth morphism of \canal spaces.
	Then there is an equivalence of categories
	\[ \mathrm{Loc}(X / S) \simeq \Coh^{\mathrm{fc}}(X / S) , \]
	where the left hand side denotes the category of local systems on $X$ relative to $S$ and the right hand side denotes the category of coherent sheaves on $X$ equipped with a flat connection which is $p\inv \cO_S$-linear.
\end{thm-intro}

One of the fundamental insights of C.\ Simpson \cite{Simpson_Moduli_representations_I,Simpson_Moduli_representations_II} is that one can rewrite the above correspondence as an equivalence of certain \emph{mapping stacks}.
More precisely, C.\ Simpson remarked that given an smooth \canal space $X$, we can associate to it a variety of different ``shapes'' that encode different aspects of the geometry of $X$.
By ``shape'', we simply mean sheaves
\[ \Stn\op \to \cS , \]
where $\cS$ is the $\infty$-category of spaces, and $\Stn$ is the category of Stein spaces (equipped with the usual Grothendieck topology generated by open immersions).
Two of the most significant shapes associated to $X$ are the \emph{Betti shape}, denoted $X\B$, and the \emph{de Rham shape}, denoted $X\DR$.
We refer to \cref{sec:shapes} for a precise definition of these objects.
For the moment, we content ourselves of recalling their characteristic properties:
\begin{enumerate}
	\item on one hand, natural transformations from $X\B$ to the analytic stack in categories $\mathbf{AnVect}_n$ of holomorphic vector bundles of rank $n$ corresponds to representations of the first fundamental groupoid $\pi_1(X\topl)$ of $X$ in $\mathbb C^n$, where $X\topl$ denotes the underlying topological space of $X$.
	\item On the other hand, if $n = \dim_{\mathbb C}(X)$, then natural transformations from $X\DR$ to the Eilenberg-Maclane stack $\mathrm K(\mathbf A^1_{\mathbb C}, n)$ can be naturally arranged into a complex of abelian groups that computes the de Rham cohomology of $X$.
\end{enumerate}
Another remarkable property of $X\DR$ that is relevant for our purposes is the following: natural transformations from $X\DR$ to $\mathbf{AnVect}_n$ form a category that is equivalent to the category of holomorphic vector bundles on $X$ equipped with a flat connection.

With this terminology, we can rephrase the absolute version of \cref{thm:Deligne} as follows: there is a natural equivalence
\begin{equation} \label{eq:RH_reformulation_I}
	\Map_{\mathrm{AnSt}_{\mathbb C}^{\mathrm{cat}}}( X\B,  \mathbf{AnVect}_n) \simeq \Map_{\mathrm{AnSt}_{\mathbb C}^{\mathrm{cat}}}(X\DR, \mathbf{AnVect}_n) .
\end{equation}
Here $\mathrm{AnSt}_{\mathbf C}^{\mathrm{cat}}$ denotes the $\infty$-category of sheaves on $\Stn$ with values in ($\infty$-)categories.

\begin{rem}
	The language of higher category theory is unavoidable in the above formulation of \cref{thm:Deligne}.
	Indeed, if $X\DR$ is simply a sheaf of sets, we will see that the Betti shape $X\B$ is a sheaf that takes genuinely values in $\cS$.
	More precisely, $X\B$ is the constant stack associated to the topological space $X\topl$.
	This should not surprise: local systems encode topological (and homotopy theoretic) aspects of the geometry of $X$, while $\cD$-modules encode the complex geometry of $X$.
\end{rem}

Having the formulation \eqref{eq:RH_reformulation_I} at our disposal, two natural questions arise naturally:
\begin{enumerate}
	\item is it possible to recover the full \cref{thm:Deligne} in this language? That is, can we formulate a relative version of \eqref{eq:RH_reformulation_I}?
	\item Can we substitute $\mathbf{AnVect}_n$ with some other analytic stack?
\end{enumerate}
It is easy to answer to the first question.
Namely, it is sufficient to replace the mapping space $\Map_{\mathrm{AnSt}_{\mathbb C}^{\mathrm{cat}}}$ with the (underived) \emph{mapping stack}:
\begin{equation}
	\bfMap_{\mathrm{AnSt}_{\mathbb C}^{\mathrm{cat}}}( X\B, \mathbf{AnVect}_n) \simeq \bfMap_{\mathrm{AnSt}_{\mathbb C}^{\mathrm{cat}}}(X\DR, \mathbf{AnVect}_n) .
\end{equation}
When evaluated on $U \in \Stn$, the above equivalence yields
\[ \Map_{\mathrm{AnSt}_{\mathbb C}^{\mathrm{cat}}}( U \times X\B,  \mathbf{AnVect}_n) \simeq \Map_{\mathrm{AnSt}_{\mathbb C}^{\mathrm{cat}}}( U \times X\DR, \mathbf{AnVect}_n) , \]
and the left hand side can be identified with local systems on $U \times X$ relative to $U$, while $U \times X\DR$ can be identified with vector bundles on $U \times X$ equipped with a $p\inv \cO_U$-linear flat connection.

The second question is more subtle.
Ideally, we would like to replace $\mathbf{AnVect}_n$ with a stack $T$ such as the analytic stack of perfect complexes $\bfAnPerf$, or the stack of $G$-bundles $\rB G$, for some reductive algebraic group $G$.
However, in these cases, it is not so clear how to construct an explicit map
\[ \Map_{\mathrm{AnSt}_{\mathbb C}^{\mathrm{cat}}}( X\B,  T) \to \Map_{\mathrm{AnSt}_{\mathbb C}^{\mathrm{cat}}}(X\DR, T) . \]
The reason is that, classically, the Riemann-Hilbert correspondence is constructed in a rather explicit way: given a vector bundle on $X$ equipped with a flat connection, the associated local system is simply the sheaf of horizontal sections of $X$.
Such a construction would still make sense for $T = \rB G$, but already in the case of $T = \bfAnPerf$ to explicitly construct such a map by hand would involve a lot of hard work with explicit model categories to handle the higher coherences.

Luckily, there is a much better way to proceed.
Indeed, a second key insight of C.\ Simpson is that there exists a natural transformation
\[ \eta_{\mathrm{RH}} \colon X\DR \to X\B \]
that induces the Riemann-Hilbert correspondence.
See \cref{lem:alternative_Betti_form} and the subsequent discussion for the construction of $\eta_{\mathrm{RH}}$ and see \cref{prop:comparison_RH} for a proof that the induced transformation
\[ \eta_{\mathrm{RH}}^* \colon \bfMap_{\mathrm{AnSt}_{\mathbb C}^{\mathrm{cat}}}( X\B, \mathbf{AnVect}_n) \to \bfMap_{\mathrm{AnSt}_{\mathbb C}^{\mathrm{cat}}}(X\DR, \mathbf{AnVect}_n) \]
coincides with the classical Riemann-Hilbert map (at least in the special case where $X\topl$ is weakly contractible).

It is the existence of the map $\eta_{\mathrm{RH}}$ that really opens the door to the generalizations of \cref{thm:Deligne} taken into consideration in this paper, as we are going to see.
The first result we obtain is the following:

\begin{thm-intro}[{cf.\ \cref{cor:RH_perfect_modules_ordinary_base}}]
	Let $X$ be a smooth \canal space.
	There is an equivalence of underived mapping stacks
	\[ \eta_{\mathrm{RH}}^* \colon \bfMap_{\mathrm{AnSt}_{\mathbb C}^{\mathrm{cat}}}( X\B, \bfAnPerf) \simeq \bfMap_{\mathrm{AnSt}_{\mathbb C}^{\mathrm{cat}}}(X\DR, \bfAnPerf) . \]
\end{thm-intro}

\begin{rem}
	When replacing $\mathbf{AnVect}_n$ by $\bfAnPerf$, natural transformations from $X\B$ no longer correspond to representations of the first fundamental groupoid of $X$, but rather to representations of the \emph{full homotopy type} of $X$.
\end{rem}

The second generalization we pursue involves the use of derived \canal geometry.
This framework is available after the foundational works \cite{DAG-IX,Porta_DCAGI,Porta_DCAGII}.
We refer to \cref{sec:review_dan} for a review of the language and of the main results needed in this paper.
Working within derived geometry has many advantages.
Among others, deformation theory becomes richer and more transparent, as the formal moduli problem correspondence \cite{DAG-X} shows.
For this reason, it is natural to ask whether the relative Riemann-Hilbert correspondence stands true when the family is parametrized by some derived analytic space.
The second main result of this paper is to show that the answer is affirmative:

\begin{thm-intro}[{cf.\ \cref{thm:derived_RH}}]
	Let $X$ be a smooth \canal space.
	Then there is an equivalence of (derived) mapping stacks
	\[ \eta_{\mathrm{RH}}^* \colon \bfMap_{\dAnSt^{\mathrm{cat}}}( X\B, \bfAnPerf) \simeq \bfMap_{\dAnSt^{\mathrm{cat}}}(X\DR, \bfAnPerf) . \]
	Moreover, the above equivalence can be promoted to an equivalence of stacks with values in \emph{symmetric monoidal} stable $\infty$-categories.
\end{thm-intro}

\begin{rem}
	A similar use of derived \canal geometry has recently appeared in \cite{DiNatale_Global_Period_2016}, where a derived version of Griffiths' period map has been constructed.
\end{rem}

\bigskip
\paragraph{\textbf{Notations}}

In this paper we use the language of $\infty$-categories.
All functors are $\infty$-functors unless specified.
We reserve the notation $\cS$ for the $\infty$-category of spaces.
Given an $\infty$-category $\cC$, we denote by $\PSh(\cC)$ the $\infty$-category of presheaves with values in $\cS$.
Our reference for $\infty$-category theory is \cite{HTT}.
We make also an extensive use of $\infty$-Grothendieck sites.
For those, we refer to \cite[\S 2.4]{Porta_Yu_Higher_analytic_stacks_2014}.

We denote by $\mathbb A^n_{\mathbb C}$ the algebraic affine space of dimension $n$ and by $\mathbf A^n_{\mathbb C}$ the analytic one.
We denote by $\Stn$ the category of Stein spaces.

\bigskip
\paragraph{\textbf{Acknowledgments.}}

I would like to thank Gabriele Vezzosi, Bertrand To\"en and Carlos Simpson for explaining the importance that the results in this paper has in relation to non-abelian Hodge theory, as well for many inspiring and clarifying discussions.
I would also like to thank Jorge Ant\'onio, Justin Hilburn, Tony Pantev, François Petit and Michel Vaquié for useful discussions.
Finally, I am grateful to Damien Calaque and Marco Robalo for having read a preliminary version of this paper.

This research was partially conducted during the period the author was supported by Simons Foundation grant number 347070 and the group GNSAGA.

\section{Review of derived \canal geometry} \label{sec:review_dan}

Derived \canal geometry plays a central role in this paper.
Since it is a relatively new framework, it is worth review the key facts.
The main references are \cite{DAG-IX,Porta_DCAGI,Porta_DCAGII,Porta_Yu_Representability}.

\begin{defin}
	We let $\cTan$ denote the (ordinary) category whose objects are Stein open subsets of $\mathbb C^n$ and whose morphisms are holomorphic functions.
	We endow $\cTan$ with the analytic topology, denoted $\tauan$.
	A morphism in $\cTan$ is said to be \emph{admissible} if it is an open immersion.
\end{defin}

\begin{defin}
	Let $\cX$ be an $\infty$-topos.
	A \emph{$\cTan$-structure} (also called an \emph{analytic ring}) is a functor
	\[ \cO \colon \cTan \to \cX \]
	such that:
	\begin{enumerate}
		\item $\cO$ commutes with products and pullbacks along admissible morphisms;
		\item $\cO$ takes $\tauan$-covers to effective epimorphisms in $\cX$.\footnote{A morphism $f \colon X \to Y$ in $\cX$ is an effective epimorphism if $\pi_0(f)$ is an epimorphism in the $1$-topos $\tau_{\le 0} \cX$.}
	\end{enumerate}
	The $\infty$-category of $\cTan$-structures in $\cX$ is denoted $\mathrm{Str}_{\cTan}(\cX)$.
	A morphism $f \colon \cO \to \cO'$ of $\cTan$-structures is said to be \emph{local} if for every admissible morphism $\varphi \colon U \to V$ in $\cTan$, the square
	\[ \begin{tikzcd}
		\cO(U) \arrow{r} \arrow{d} & \cO(V) \arrow{d} \\
		\cO'(U) \arrow{r} & \cO'(V)
	\end{tikzcd} \]
	is a pullback square.
	The $\infty$-category of $\cTan$-structures and local morphisms between them is denoted $\mathrm{AnRing}(\cX)$, and we refer to it as the $\infty$-category of analytic rings in $\cX$.
\end{defin}

\begin{eg} \label{eg:simple_analytically_structured_topos}
	Let $X$ be a \canal space and let $X\topl$ denote the underlying topological space of $X$.
	Let $\cX \coloneqq \Sh(X\topl)$ be the $\infty$-topos of sheaves on $X\topl$.
	We can define an analytic ring $\cO$ on $\cX$ as follows:
	\[ \cO \colon \cTan \to \cX \]
	is by definition the functor that sends $U \in \cTan$ to the sheaf $\cO(U)$ on $X\topl$ defined by
	\[ \mathrm{Op}(X\topl) \ni V \mapsto \cO(U)(V) \coloneqq \Hom_{\An}(V, U) \in \rSet , \]
	where $\mathrm{Op}(X\topl)$ denotes the lattice of open subsets of $X\topl$.
	Notice that $\cO(\mathbf A^1_{\mathbb C})$ coincides with the sheaf of holomorphic functions on $X$.
\end{eg}

\begin{rem}
	Let $\cX = \cS$ be the $\infty$-category of spaces and let $A \in \mathrm{AnRing}_{\mathbb C} \coloneqq \mathrm{AnRing}(\cS)$.
	Evaluating $A$ on $\mathbf A^1$ produces a \emph{simplicial commutative $\mathbb C$-algebra}.
	This simplicial commutative algebra is denoted $A\alg$.
	This operation gives a simple way of understanding analytic rings: they simply are simplicial commutative algebras equipped with an extra structure.
	This extra structure amounts to a number of additional operations that are allowed on the ring - one for every holomorphic function $f \colon U \to \mathbb C$, where $U \subset \mathbb C^n$. For example, we can take the exponential of an element in $A$, or if $U \subseteq \mathbb C \smallsetminus \{0\}$ is a simply connected open subset and $x \in A(U)$, we can evaluate a branch of the complex logarithm on $x$.
	This extra structure plays the role of a topology on $A$ (one that would make $A$ into a Banach or Fréchet algebra), but it is easier to deal with at the categorical level, and it is well suited for the derived setting.
\end{rem}

\begin{defin}
	A \emph{derived \canal space} is a pair $(\cX, \cO_X)$, where $\cX$ is an $\infty$-topos and $\cO_X$ is an analytic ring in $\cX$, such that:
	\begin{enumerate}
		\item locally on $\cX$,\footnote{Formally speaking, this means that there exists a collection of objects $\{U_i\}$ of $\cX$ such that $\coprod U_i \to \mathbf 1_\cX$ is an effective epimorphism and such that the structured topos $(\cX_{/U_i}, \cO_X |_{U_i})$ satisfies the given condition.} $(\cX, \pi_0 \cO_X)$ is equivalent to a structured topos arising from the construction of \cref{eg:simple_analytically_structured_topos};
		\item the sheaves $\pi_i(\cO_X\alg)$ are coherent as sheaves of $\pi_0(\cO_X\alg)$-modules.
	\end{enumerate}
\end{defin}

\begin{rem}
	The reader that is unfamiliar with the language of $\infty$-topoi can simply think in terms of topological spaces. The above definition is meant to capture more than ordinary \canal spaces: it gives a way to deal uniformly with \canal spaces \emph{and} \canal orbifolds.
\end{rem}

This is the setup for derived \canal geometry. I collect below the results obtained so far in this context that are relevant for the current paper:

\begin{thm}[{\cite[\S 12]{DAG-IX}}]
	Derived \canal spaces assemble in an $\infty$-category $\dAn_{\mathbb C}$ that enjoys the following property:
	\begin{enumerate}
		\item fiber products in $\dAn_{\mathbb C}$ exists;
		\item the category of (ordinary) \canal spaces embeds fully faithfully in $\dAn_{\mathbb C}$, the embedding given on objects by the construction performed in \cref{eg:simple_analytically_structured_topos}.
	\end{enumerate}
\end{thm}

\begin{thm}[{\cite{Porta_DCAGII}}] \label{thm:equivalence_modules}
	Let $X = (\cX, \cO_X)$ be a derived \canal space.
	There is a canonical equivalence
	\[ \Sp(\AnRing(\cX)_{/\cO_X}) \simeq \cO_X\alg \Mod . \]
\end{thm}

\begin{rem}
	The stable $\infty$-category $\cO_X\alg \Mod$ has a canonical $t$-structure. We refer to \cite[Proposition 1.7]{DAG-VII} for the construction.
\end{rem}

\begin{defin}
	Let $X = (\cX, \cO_X)$ be a derived \canal space.
	We let $\Coh^-(X)$ be the full stable subcategory of $\cO_X\alg \Mod$ spanned by those sheaves $\cF$ such that $\rH^i(\cF)$ is a coherent sheaf of $\pi_0(\cO_X\alg)$-modules for every $i$ and $\rH^i(\cF) = 0$ for $i \gg 0$.
	We refer to $\Coh^-(X)$ as the $\infty$-category of \emph{almost perfect modules} over $X$.
	The $\infty$-category of \emph{perfect modules} $\Perf(X)$ is the full subcategory of $\Coh^-(X)$ spanned by those object that, locally on $X$, belong to the smallest stable full subcategory of $\Coh^-(X)$ closed under retracts and containing $\cO_X\alg$.
\end{defin}

\begin{rem}
	Both assignments $X \mapsto \Coh^-(X)$ and $X \mapsto \Perf(X)$ can be promoted to sheaves with values in (stable) $\infty$-categories.
	We denote such stacks by
	\[ \bfAnCoh^-, \bfAnPerf \colon \dStn\op \to \Cat_\infty . \]
	One can prove the deduce the descent property of $\bfAnPerf$ out of the one for $\bfAnCoh^-$.
	For the latter, the strategy of reduction to the heart (see \cite[\S 6]{DAG-VII}) applies in the analytic setting as well.
	It can be shown that $\bfAnPerf$ is the analytification of the algebraic stack of perfect complexes $\bfAnPerf$ (as stack in symmetric monoidal stable $\infty$-categories).
	This is a consequence of the GAGA theorems of \cite{Porta_DCAGI}, and the proofs will appear in \cite{DiNatale_Analytification_mapping_stacks}.
	In particular, it follows from \cite{Toen_Moduli} and \cite[Proposition 2.25]{Porta_Yu_Higher_analytic_stacks_2014} that the associated stacks in $\infty$-groupoids is locally geometric.
\end{rem}

The importance of \cref{thm:equivalence_modules} is that it allows to study derived deformation theory in the \canal setting:

\begin{defin}
	Let $X = (\cX, \cO_X)$ be a derived \canal space and let $\cF \in \Coh^{\le 0}(X)$.
	The \emph{split square-zero extension of $\cO_X$ by $\cF$} is the analytic ring
	\[ \cO_X \oplus \cF \coloneqq \Omega^\infty(\cF) \in \AnRing(\cX)_{/\cO_X} . \]
	The \emph{space of (derived) derivations of $\cO_X$ with values in $\cF$} is by definition the space
	\[ \mathrm{Der}\an(\cO_X; \cF) \coloneqq \Map_{\AnRing(\cX)_{/\cO_X}}( \cO_X, \cO_X \oplus \cF ) . \]
	Let $d \colon \cO_X \to \cO_X \oplus \cF$ be a derived derivation.
	The \emph{analytic square-zero extension associated to $d$} is by definition the pullback
	\[ \begin{tikzcd}
		\cO_X \oplus_d \cF \arrow{r} \arrow{d} & \cO_X \arrow{d}{d} \\
		\cO_X \arrow{r}{d_0} & \cO_X \oplus \cF ,
	\end{tikzcd} \]
	where $d_0$ denotes the zero derivation.
	An \emph{analytic square-zero extension} of $\cO_X$ is a morphism of analytic rings $\cA \to \cO_X$ that arises by the above construction.
\end{defin}

The following theorem summarizes the main properties of derived analytic deformation theory:

\begin{thm}[{\cite{Porta_Yu_Representability}}]
	Let $X = (\cX, \cO_X)$ be a derived \canal space.
	\begin{enumerate}
		\item There exists an \emph{analytic cotangent complex} $\mathbb L_X\an$ that classifies analytic derivations.
		\item There exists a \emph{relative analytic cotangent complex} attached to a map $f \colon X \to Y$ of derived \canal spaces. A sequence of maps $X \xrightarrow{f} Y \xrightarrow{g} Z$ gives rise to a fiber sequence
		\[ f^* \mathbb L\an_{Y/Z} \to \mathbb L\an_{X/Z} \to \mathbb L\an_{X / Y} . \]
		\item The morphisms $\tau_{\le n+1} \cO_X \to \tau_{\le n} \cO_X$ are analytic square-zero extensions.
		\item Let $\cF \in \Coh^{\le -1}(X)$ and let $d \colon \cO_X \to \cO_X \oplus \cF$ be an analytic derivation.
		Then the pair $(\cX, \cO_X \oplus_d \cF)$ is a derived \canal space, denoted $X_d[\cF]$.
	\end{enumerate}
\end{thm}

Before passing to the main body of this paper, there is one last important idea that needs to be discussed: the unramifiedness of $\cTan$.
While \emph{per se} this is a technical notion (we refer to \cite[\S 1]{DAG-IX} for a thorough discussion), it has a number of important consequences that are easy to formulate and that are extremely useful in practice.
These properties are related to the notion of closed immersion, so let us recall it here:

\begin{defin}
	A morphism $f \colon X = (\cX, \cO_X) \to Y = (\cY, \cO_Y)$ of derived \canal spaces is said to be a \emph{closed immersion} if:
	\begin{enumerate}
		\item the underlying morphism of $\infty$-topoi $f\inv \colon \cY \leftrightarrows \cX \colon f_*$ is a closed immersion of $\infty$-topoi;\footnote{The reader unfamiliar with the language of ($\infty$-)topoi may think in terms of topological spaces. Its intuition will be correct. The curious reader is referred to \cite[\S 7.3.2]{HTT}.}
		\item the induced morphism $f\inv \cO_Y \to \cO_X$ is an effective epimorphism, in the sense that for any $U \in \cTan$ the morphism $f\inv \cO_Y(U) \to \cO_X(U)$ is an effective epimorphism in $\cX$.
	\end{enumerate}
\end{defin}

\begin{eg}
	Let $X = (\cX, \cO_X)$ be a derived \canal space and let $\cF \in \Coh^{\le -2}(\cF)$.
	Let $d \colon \cO_X \to \cO_X \oplus \cF$ be an analytic derivation.
	The canonical morphism $X \to X_d[\cF]$ is a closed immersion.
\end{eg}

\begin{thm}[{\cite[\S 11]{DAG-IX}}] \label{thm:unramifiedness}
	Let
	\[ \begin{tikzcd}
		X' \arrow[hook]{r}{i} \arrow{d}{g} & Y' \arrow{d}{f} \\
		X \arrow[hook]{r}{j} & Y
	\end{tikzcd} \]
	be a pullback diagram in $\dAn_{\mathbb C}$, where $j$ is a closed immersion.
	Then:
	\begin{enumerate}
		\item the underlying algebra of the structure sheaf of $X'$ can be computed as the tensor product
		\[ \cO_{Y'} \otimes_{f\inv \cO_Y} f\inv j_* \cO_X . \]
		\item For any $\cF \in \cO_X \Mod$, the natural transformation
		\[ f^* j_*(\cF) \to i_* g^*(\cF) \]
		is an equivalence.
	\end{enumerate}
\end{thm}

Finally, a word about analytic stacks with values in $\infty$-categories.
We denote by
\[ \dAnSt^{\mathrm{cat}} \coloneqq \Sh_{\Cat_\infty}( \dStn, \tauan ) \]
the $\infty$-category of sheaves on $(\dStn, \tauan)$ with values in $\Cat_\infty$.
Notice that there is a canonical embedding $\dAnSt \hookrightarrow \dAnSt^{\mathrm{cat}}$ that is induced by the natural inclusion $\cS \hookrightarrow \Cat_\infty$ of $\infty$-groupoids into $\infty$-categories.
Notice also that $\dAnSt^{\mathrm{cat}}$ has a (cartesian) symmetric monoidal structure and that it is closed with respect to this monoidal structure.
Furthermore, the inclusion $\dAnSt \hookrightarrow \dAnSt^{\mathrm{cat}}$ is strong monoidal and it commutes with the internal hom.
For this reason, we denote the internal hom in $\dAnSt^{\mathrm{cat}}$ simply as $\bfMap$.

\section{De Rham and Betti shapes} \label{sec:shapes}

We denote by $\dStn$ the $\infty$-category of derived Stein spaces, as introduced in \cite{Porta_DCAGI}.
We further denote by $\Stnred$ the category of reduced (discrete) Stein spaces.
Finally, we let
\[ j \colon \Stnred \hookrightarrow \dStn \]
denote the natural fully faithful inclusion.
The analytic topology $\tauan$ on $\dStn$ restricts to a Grothendieck topology on $\Stnred$, which we still denote $\tauan$.
We let $\dAnSt \coloneqq \St( \dStn, \tauan )$ be the $\infty$-topos of derived analytic stacks.
The following elementary observation will be useful later on:

\begin{lem} \label{lem:Stein_reduced_cocontinuous}
	The functor $j \colon \Stnred \hookrightarrow \dStn$ is both continuous and cocontinuous for the analytic topologies on these two categories.
	In particular, the restriction
	\[ j^s \colon \dAnSt \to \St( \Stnred, \tauan ) \]
	admits both a left adjoint and a right adjoint.
\end{lem}

\begin{proof}
	We refer to \cite[\S 2.4]{Porta_Yu_Higher_analytic_stacks_2014} for the terminology of morphisms of $\infty$-sites and for the choice of notations.
	It is clear that the functor $j$ takes $\tauan$-coverings to $\tauan$-coverings, and that it commutes with pullbacks along open immersions.
	It follows that $j$ is continuous.
	To see why it is cocontinuous, fix $X \in \Stnred$ and let $\{U_i \to X\}$ be a $\tauan$-cover in $\dStn$.
	Since the maps $U_i \to X$ are local biholomorphisms, we see that each $U_i$ has to be reduced on its own.
	In particular, $\{U_i \to X\}$ is a $\tauan$-cover in $\Stnred$. Thus $j$ is cocontinuous.
	
	The last assertion now follows from \cite[Lemmas 2.14 \& 2.21]{Porta_Yu_Higher_analytic_stacks_2014}.
\end{proof}

Following \cite{CPTVV}, we give the following definition:

\begin{defin}
	\begin{enumerate}
		\item The \emph{De Rham functor} is the functor
		\[ (-)\DR \colon \dAnSt \to \dAnSt \]
		defined as the composition
		\[ \begin{tikzcd}
			\dAnSt \arrow{r}{j^s} & \St( \Stnred, \tauan ) \arrow{r}{\Lan_j} & \dAnSt ,
		\end{tikzcd} \]
		where $\Lan_j$ denotes the left Kan extension along $j$.
		\item The \emph{reduction functor} is the functor
		\[ (-)\red \colon \dAnSt \to \dAnSt \]
		defined as the composition
		\[ \begin{tikzcd}
			\dAnSt \arrow{r}{j^s} & \St( \Stnred, \tauan ) \arrow{r}{\mathrm{Ran}_j} & \dAnSt ,
		\end{tikzcd} \]
		where $\mathrm{Ran}_j$ denotes the right Kan extension along $j$.
	\end{enumerate}
\end{defin}

As in the algebraic setting, given $F \in \dAnSt$, we can characterize $F\DR$ as the composition $F \circ j \circ (-)\red$.
In other words, for every $U \in \dStn$, we have:
\[ F\DR(U) \simeq F( \trunc(U)\red ) . \]

Let us record the following handy consequence of \cref{lem:Stein_reduced_cocontinuous}:

\begin{cor} \label{cor:De_Rham_colimits}
	The De Rham functor $(-)\DR \colon \dAnSt \to \dAnSt$ commutes with finite products and with arbitrary colimits.
	The reduced functor $(-)\red \colon \dAnSt \to \dAnSt$ commutes with arbitrary limits.
\end{cor}

\begin{proof}
	Let us prove the assertion for $(-)\DR$; the one for $(-)\red$ can be dealt with in a similar way.
	By definition, $(-)\DR$ can be written as the composition $\Lan_j(-) \circ j^s$.
	The functor $\Lan_j(-)$ is the left adjoint of $j^s$.
	In particular, it commutes with arbitrary colimits.
	We now conclude by observing that $j^s \colon \dAnSt \to \St(\Stnred, \tauan)$ commutes with colimits in virtue of \cref{lem:Stein_reduced_cocontinuous}.
	Moreover, since the functor $j \colon \Stnred \hookrightarrow \dStn$ commutes with finite products, we see that $\Lan_j$ also commutes with finite products.
	Hence $(-)\DR$ commutes with finite products.
\end{proof}

Let now $X \in \dAn_{\mathbb C}$ be a derived \canal space.
Suppose that $\trunc(X)$ is an usual \canal space (in other words, the underlying $\infty$-topos of $X$ is $0$-localic).
We let $X\topl$ be the underlying topological space of $X$.
Observe that this construction provides us with a functor
\[ (-)\topl \colon \dStn \to \cS . \]

\begin{defin}
	Let $X \in \dAn_{\mathbb C}$ be as above.
	The \emph{Betti form} of $X$ is the constant stack
	\[ X\B \colon \dStn\op \to \cS \]
	associated to $X\topl$.
\end{defin}

\begin{rem}
	Observe that neither $X\DR$ nor $X\B$ are sensitive to the derived structure on $X$.
	Indeed, since $X$ and $\trunc(X)$ share the same underlying $\infty$-topos, we have a canonical equivalence
	\[ X\B \simeq (\trunc(X))\B . \]
	On the other side, the universal property of the truncation shows that for every $U \in \dStn$, we have
	\[ \Map_{\dAn_{\mathbb C}}(U\red, X) \simeq \Map_{\dAn_{\mathbb C}}(U\red, \trunc(X)) . \]
	It follows that there is a canonical equivalence $X\DR \simeq (\trunc(X))\DR$.
\end{rem}

Fix a derived \canal space $X \in \dAn_{\mathbb C}$ whose underlying $\infty$-topos is $0$-localic.
As first observed by C.\ Simpson, there is a remarkable morphism of stacks
\[ X\DR \to X\B , \]
which we are going to construct.

\begin{lem} \label{lem:alternative_Betti_form}
	The Betti form $X\B$ is naturally equivalent to the $\infty$-functor
	\begin{equation} \label{eq:alternative_Betti_form}
		\Map_\cS( (-)\topl, X\topl ) \colon \dStn\op \to \cS .
	\end{equation}
\end{lem}

\begin{proof}
	Denote by $\widetilde{X\B}$ the constant presheaf
	\[ \widetilde{X\B} \colon \dStn\op \to \cS \]
	associated to $X\topl$.
	In this way, $X\B$ is the sheafification of $\widetilde{X\B}$.
	Denote further by $F$ the functor \eqref{eq:alternative_Betti_form}.
	We have
	\[ \Map_{\dAnSt}(X\B, F) \simeq \Map_{\PSh(\dStn)}(\widetilde{X\B}, F) . \]
	Let $P \coloneqq \{*\}$ be the terminal category and consider the canonical functor
	\[ \pi \colon \dStn \to P . \]
	The induced functor
	\[ \pi_p \colon \cS \simeq \PSh(P) \to \PSh(\dStn) \]
	takes $X\topl$ to $\widetilde{X\B}$.\footnote{We refer to \cite[\S 2.4]{Porta_Yu_Higher_analytic_stacks_2014} and to \cite{stacks-project} for the choice of the notations.}
	As a consequence, we see that
	\[ \Map_{\PSh(\dStn)}(\widetilde{X\B}, F) \simeq \Map_{\PSh(\dStn)}(\pi_p (X\topl), F) \simeq \Map_{\cS}( X\topl, \mathrm{Ran}_\pi(F) ) . \]
	Since $\dStn$ has a final object (given by $*_{\mathbb C}$, the point seen as a \canal space), we obtain:
	\[ \mathrm{Ran}_\pi(F) \simeq F(*_{\mathbb C}) \simeq X\topl . \]
	Therefore, the identity $X\topl \xrightarrow{\mathrm{id}} X\topl$ induces a map
	\[ \widetilde{\varphi} \colon \widetilde{X\B} \to F . \]
	The universal property of sheafification produces in turn a map
	\[ \varphi \colon X\B \to F . \]
	We claim that this map is an equivalence.
	
	To see this, introduce the full subcategory $\dStn^{\mathrm{wc}}$ spanned by those derived Stein spaces $U \in \dStn$ such that $U\topl$ is weakly contractible.
	Since every $V \in \dStn$ admits an hypercover by objects in $\dStn^{\mathrm{wc}}$, it is enough to check that $\varphi$ is an equivalence on the elements of $\dStn^{\mathrm{wc}}$.
	
	Observe now that, for any $U \in \dStn^{\mathrm{wc}}$, we have
	\[ \widetilde{X\B}(U) \simeq X\topl, \qquad F(U) \simeq \Map_{\cS}(U\topl, X\topl) , \]
	and the map $\widetilde{\varphi}_U$ corresponds, via adjunction, to the canonical projection
	\[ X\topl \times U\topl \to X\topl . \]
	Since $U\topl$ is contractible in $\cS$, we conclude that $\widetilde{\varphi}_U$ is an equivalence.
	It follows that $\varphi_U$ is an equivalence as well.
	The proof is now complete.
\end{proof}

Let again $X \in \dAn_{\mathbb C}^0$ be a derived \canal space whose underlying $\infty$-topos is $0$-localic.
We review $X$ as an object in $\dAnSt$ via its functor of points.
Observe that
\[ X\DR = \Map_{\dAn_{\mathbb C}^0}((-)\red, X) . \]
This provides us with a natural transformation
\[ \Map_{\dAn_{\mathbb C}^0}((-)\red, X) \to \Map_{\cS}(((-)\red)\topl, X\topl) \simeq \Map_{\cS}((-)\topl, X\topl) . \]
Using \cref{lem:alternative_Betti_form}, we can identify the right hand side with $X\B$.
This produces a map $\eta_{\mathrm{RH}} \colon X\DR \to X\B$.
We refer to $\eta_{\mathrm{RH}}$ to as the \emph{Riemann-Hilbert transformation}.
The goal of this paper is to show that precomposition with $\eta_{\mathrm{RH}}$ induces an equivalence of derived analytic stacks (see \cref{thm:derived_RH}):
\begin{equation} \label{eq:derived_RH}
	\eta_{\mathrm{RH}}^* \colon \bfMap(X\B, \bfAnPerf) \to \bfMap(X\DR, \bfAnPerf) .
\end{equation}

\begin{rem} \label{rem:explaining_why_analytic_topology}
	The map $\eta_{\mathrm{RH}}$ arises naturally as a zig-zag
	\[ X\DR \to \Map_{\cS}((-)\topl, X\topl) \leftarrow X\B . \]
	This same zig-zag exists when working over the category of affine derived schemes $\mathrm{dAff}_{\mathbb C}$.
	However, in this situation \cref{lem:alternative_Betti_form} fails and therefore we cannot invert the morphism on the right.
	In other words, the construction of $\eta_{\mathrm{RH}}$ explicitly requires the analytic topology.
\end{rem}

Before starting to prove that the map \eqref{eq:derived_RH} is an equivalence, a number of preliminaries are necessary.
As first step, we need to develop some techniques in order to deal with the stack $X\DR$.

\section{Formal completions in analytic geometry}

In algebraic geometry, the stack $X\DR$ is intrinsically related to formal geometry.
More precisely, the canonical map $p \colon X \to X\DR$ can be thought of as parametrizing a family of formal schemes mapping to $X$.
See \cite[\S 2]{CPTTV} for careful explanations of this idea.
We recall here two basic results that will be needed later:

\begin{lem} \label{lem:De_Rham_analytic}
	Let $X$ be a smooth analytic space and let $p \colon X \to X\DR$ be the canonical morphism.
	Then:
	\begin{enumerate}
		\item the map $p$ is an effective epimorphism in $\dAnSt$. In particular, if $X^\bullet / X\DR$ denotes the \v{C}ech nerve of $p$, then there is an equivalence
		\[ |X^\bullet / X\DR| \simeq X\DR \]
		inside $\dAnSt$.
		
		\item Let $\Delta_X^{(n)}$ be the $n$-th infinitesimal neighbourhood of the diagonal in $X \times X$.
		Then there is an equivalence
		\[ \colim_n \Delta_X^{(n)} \simeq X \times_{X\DR} X , \]
		the colimit being computed in $\PSh(\dStn)$.
	\end{enumerate}
\end{lem}

\begin{proof}
	We start by proving (1). It is enough to prove that $p$ is an effective epimorphism.
	Let therefore $U \in \dStn$ and let $f \colon U \to X\DR$ be any morphism.
	We have to show that, up to covering $U$, the morphism $f$ can be lifted to a map $\overline{f} \colon U \to X$.
	This question is local on $X$, and therefore we can assume that $X$ is biholomorphic to an open ball in $\mathbf A^n_{\mathbb C}$.
	Recall that, by definition, the map $f \colon U \to X\DR$ corresponds to a map $U\red \to X$.
	The composition
	\[ g \colon U\red \to X \hookrightarrow \mathbb A^n_{\mathbb C} \]
	is equivalent to the choice of $n$ holomorphic functions on $U\red$, that we denote $g_1, \ldots, g_n$.
	Let $\cO_U$ be the structure sheaf of $U$ and let $\cJ$ denote the nilradical sheaf, so that there is a short exact sequence
	\[ 0 \to \cJ \to \cO_U \to \cO_{U\red} \to 0 . \]
	Since $U$ is Stein, passing to global sections and invoking Cartan's theorem B we obtain
	\[ \Gamma(U; \cO_{U\red}) \simeq \Gamma(U, \cO_U) / \Gamma(U; \cJ) . \]
	In particular, we can find holomorphic functions $\overline{f_1}, \ldots, \overline{f_n}$ on $U$ that represent the functions $g_1, \ldots, g_n$ on $U\red$.
	These new functions on $U$ define a map
	\[ \overline{f} \colon U \to \mathbf A^n_{\mathbb C} . \]
	Since the map $X \to \mathbf A^n_{\mathbb C}$ is an open immersion, we know that $\overline{f}$ factors through $X$ if and only if it factors on the topological level.
	However, since $U\topl = (U\red)\topl$ and since the map $\overline{f}$ is a lifting of the map $g \colon U\red \to X \to \mathbf A^n_{\mathbb C}$, we see that $\overline{f}$ factors topologically through $X$.
	This produces the lifting we were looking for.
	
	As for point (2), the same proof of \cite[Proposition 6.5.5]{Gaitsgory_Indschemes} applies.
	\personal{Thanks to unramifiedness, in dealing with closed immersions of derived analytic spaces, we can forget the analytic structure.}

\end{proof}

In virtue of the above lemma it is not surprising that, in dealing with $X\DR$ in the \canal world, we need some results concerning formal completions in \canal geometry.
The idea of mixing formal geometry and complex analytic geometry is certainly not new, and it goes back at least to the book \cite[Chapter VI]{Banica_Stanisla}.
The technical implementation is essentially straightforward, but, as usual in complex geometry, there is some subtlety due to the fact that the Stein open subsets are not compact (as opposed to the situation in rigid geometry).

The above lemma has the following useful consequences:

\begin{cor} \label{cor:De_Rham_analytic_coherent_sheaves}
	Let $X$ be a smooth analytic space.
	Then:
	\begin{enumerate}
		\item there is an equivalence of stable $\infty$-categories
		\[ \Coh^-(X\DR) \simeq \varprojlim \Coh^-(X^\bullet / X\DR) . \]
		\item There is an equivalence of stable $\infty$-categories
		\[ \Coh^-(X \times_{X\DR} X) \simeq \varprojlim_n \Coh^-(\Delta^{(n)}_X) . \]
	\end{enumerate}
\end{cor}

This description of $\Coh^-(X\DR)$ is quite useful, and we will use it several times on our way to \cref{thm:derived_RH}.
As first important application, let us prove that $\eta_{\mathrm{RH}}^*$ induces the classical Riemann-Hilbert at least in a special situation:

\begin{prop} \label{prop:comparison_RH}
	Let $X$ be a smooth \canal space and suppose that $X\topl$ is weakly contractible.
	Let $U \in \Stn$ be an ordinary Stein space.
	Then, the map $\eta_{\mathrm{RH}}$ induces a categorical equivalence
	\[ \eta_{\mathrm{RH}}^* \colon \Cohh(U \times X\B) \to \Cohh(U \times X\DR) \]
	that coincides with the classical Riemann-Hilbert correspondence of \cite[Théorème 2.23]{Deligne_Equations_differentielles}.
\end{prop}

\begin{proof}
	Since the classical Riemann-Hilbert map of \cite[Théorème 2.23]{Deligne_Equations_differentielles} is a categorical equivalence, it is sufficient to argue that $\eta_{\mathrm{RH}}^*$ coincides with that map.
	
	Observe that since $X\topl$ is weakly contractible, we have $X\B \simeq *$ and therefore $\eta_{\mathrm{RH}} \colon X\DR \to X\B$ becomes equivalent to the canonical map to the final object of $\dAnSt$.
	As a consequence, we can identify the relative Riemann-Hilbert map $\eta_{\mathrm{RH}} \colon U \times X\DR \to U \times X\B \simeq U$ with the canonical projection
	\[ p \colon U \times X\DR \to U , \]
	and similarly $U \times X \to U \times X\B \simeq U$ is identified with the projection $q \colon U \times X \to U$.
	It follows that the composition
	\[ \Cohh(U) \simeq \Cohh(U \times X\B) \to \Cohh(U \times X\DR) \to \Cohh(U \times X) \]
	is nothing but the pullback
	\[ q^* \colon \Cohh(U) \to \Cohh(U \times X) . \]
	For the same reason, the maps
	\[ \Cohh(U) \simeq \Cohh(U \times X\B) \rightrightarrows \Cohh( U \times \widehat{X}_{X \times X} ) \]
	coincide with the pullbacks along the maps
	\[ U \times \widehat{X}_{X \times X} \rightrightarrows U \times X \to U . \]
	
	On the other side, let us denote by
	\[ F \colon \Cohh(U \times X\B) \to \Cohh(U \times X\DR) \]
	the functor realizing the classical Riemann-Hilbert correspondence.
	The construction of this functor in \cite[Théorème 2.23]{Deligne_Equations_differentielles} makes it clear that the composition
	\[ \Cohh(U) \simeq \Cohh(U \times X\B) \xrightarrow{F} \Cohh(U \times X\DR) \to \Cohh(U \times X) \]
	coincides with the pullback functor $q^* \colon \Cohh(U) \to \Cohh(U \times X)$, and that the same remains true when we further pullback to $\widehat{X}_{X \times X}$.
	Finally, since $\Cohh(U \times X\DR)$ is a $1$-category, we see that
	\[ \Cohh(U \times X\DR) \simeq \mathrm{eq} \left( \Cohh(U \times X) \rightrightarrows \Cohh(U \times \widehat{X}_{X \times X}) \right) . \]
	The universal property of equalizers therefore implies that $\eta_{\mathrm{RH}}^*$ coincides with $F$.
\end{proof}

\begin{rem}
	We expect the above comparison to remain true also when $X\topl$ is not weakly contractible.
	However, the proof is bound to be more elaborated and possibly dependent on the choice of a basepoint on $X$.
	In this paper, we will only need the comparison at the level of generality we stated.
\end{rem}

We were able to exploit the descriptions provided by \cref{cor:De_Rham_analytic_coherent_sheaves} in order to prove \cref{prop:comparison_RH}.
In what follows, we will also need to manipulate $\Coh^-(X \times_{X\DR} X )$.
In this case, the description provided by \cref{cor:De_Rham_analytic_coherent_sheaves} is not entirely sufficient.
Our second goal in this section is to provide a different one.

Consider $X$ as a structured topos, $X = (\cX, \cO_X)$.
The diagonal
\[ \delta \colon X \to X \times X \]
induces a geometric morphism of $\infty$-topoi
\[ \delta\inv \colon \cX \times \cX \leftrightarrows \cX \colon \delta_* . \]
In particular, we get a sheaf of commutative rings $\delta\inv \cO_{X \times X}$ on $\cX$.
This sheaf of rings comes equipped with a sheaf of ideals $\cJ$, with the property that
\[ \delta\inv \cO_{X \times X} / \cJ \simeq \cO_X . \]
We define a new structured topos
\[ \widehat{X} \coloneqq \left(\cX, \varprojlim_n \delta\inv \cO_{X \times X} / \cJ^n \right) . \]
We will simply denote the structure sheaf of $\widehat{X}$ by $\cO_{\widehat{X}}$.
Associated to this structured topos there is a stable $\infty$-category
\[ \Coh^-(\widehat{X}) , \]
that is formally defined as the full subcategory of $\cO_{\widehat{X}} \Mod$ spanned by those sheaves that are cohomologically bounded above and whose cohomology sheaves are locally of finite presentation (this definition is correct because $\cO_{\widehat{X}}$ is coherent, see \cite[Proposition VI.2.5]{Banica_Stanisla}).

The category $\Coh^-(\widehat{X})$ admits a canonical $t$-structure.
Moreover, the morphism $\delta\inv \cO_{X \times X} \to \cO_{\widehat{X}}$ is flat, as follows for example from \cite[Lemma VI.2.2]{Banica_Stanisla} and the classical results on completions of noetherian rings.
Note that we are implicitly invoking \cite[7.3.4.9]{HTT} in order to be able to work on Stein compact subsets of $X$.
As a consequence, we obtain:

\begin{lem} \label{lem:transition_maps_t_exact}
	The canonical maps
	\[ \Coh^-(X) \rightrightarrows \Coh^-(\widehat{X}) \]
	are $t$-exact.
\end{lem}

\begin{proof}
	In virtue of the above considerations, it is enough to observe that, since $X$ is smooth, the two maps $\cO_X \rightrightarrows \delta\inv \cO_{X \times X}$ are flat.
\end{proof}

The second key result of this section is the following analogue (of a particular case of) formal GAGA:

\begin{prop} \label{prop:key_description_coherent_sheaves_formal_completion}
	There is an equivalence of stable $\infty$-categories
	\[ \Coh^-(\widehat{X}) \simeq \Coh^-(X \times_{X\DR} X ) . \]
\end{prop}

\begin{proof}
	In virtue of \cref{cor:De_Rham_analytic_coherent_sheaves} it is enough to produce an equivalence of $\Coh^-(\widehat{X})$ with
	\[ \varprojlim_n \Coh^-(\Delta^{(n)}_X) \simeq \varprojlim_n \Coh^-(\cX, \delta\inv \cO_{X \times X} / \cJ^n ) . \]
	Since both sides are sheaves with respect to the analytic topology on $X$, we can invoke \cite[7.3.4.9]{HTT} in order to reduce ourselves to check that the map is an equivalence when $X$ is replaced by a Stein compact in $X$.
	In this case, $(\cX, \cO_{X \times X} / \cJ^n)$ is again a Stein compact, and therefore we obtain equivalences
	\[ \Coh^-(\Delta^{(n)}_X) \simeq \Coh^-(\Gamma(X; \delta\inv \cO_{X \times X} / \cJ^n)) . \]
	Similarly,
	\[ \Coh^-(\widehat{X}) \simeq \Coh^-(\Gamma(X; \cO_{\widehat{X}})) . \]
	Since \cite[Lemma VI.2.2]{Banica_Stanisla} implies that
	\[ \Gamma(X; \cO_{\widehat{X}}) \simeq \varprojlim_n \Gamma(X; \delta\inv \cO_{X \times X} / \cJ) , \]
	the proposition is now a consequence of the usual formal GAGA theorem for almost perfect modules (see for example \cite[Theorem 5.3.2]{DAG-XII}).
\end{proof}

\section{$t$-structures}

Using the results of the previous section we easily get the following result:

\begin{prop} \label{prop:t-structure_relative_crystals}
	Let $X$ be a smooth analytic space and let $p \colon X \to X\DR$ be the canonical map.
	Let $X^\bullet / X\DR$ denote the \v{C}ech nerve of $p$.
	Then for any $V \in \dStn$, the face maps in the diagram
	\[ \Coh^-(V \times (X^\bullet / X\DR)) \colon \mathbf \Delta \to \Catst \]
	are $t$-exact.
	In particular, for any $V \in \dStn$, there is a $t$-structure on
	\[ \Coh^-(V \times X\DR) \simeq \varprojlim \Coh^-(V \times (X^\bullet / X\DR)) \]
	characterized by the property that the forgetful functor
	\[ \Coh^-(V \times X\DR) \to \Coh^-(V \times X) \]
	is $t$-exact.
\end{prop}

\begin{proof}
	Using \cite[7.3.4.9]{HTT} in order to work over noetherian Stein compacts inside $V$, we can prove the following relative version of \cref{prop:key_description_coherent_sheaves_formal_completion}:
	\[ \Coh^-(V \times (X \times_{X\DR} X)) \simeq \Coh^-(V \times \widehat{X}) , \]
	as well as a version for completions along higher diagonals:
	\[ \Coh^-(V \times (X^n / X\DR)) \simeq \Coh^-(V \times \widehat{X}_{X^n}) . \]
	Since the face maps in the diagram $\widehat{X}_{X^\bullet}$ are flat, we conclude that all the face maps in the diagram
	\[ \Coh^-(V \times (X^\bullet / X\DR)) \]
	are $t$-exact, just as in \cref{lem:transition_maps_t_exact}.
	In particular, invoking \cite[Lemma 3.20]{Hennion_Formal_gluing}, we obtain a canonical $t$-structure on
	\[ \Coh^-(V \times X\DR) \simeq \varprojlim \Coh^-(V \times (X^\bullet / X\DR)) \]
	having the required properties.
\end{proof}

\begin{prop} \label{prop:t_structure_Betti_form}
	For any $V \in \dStn$ there is a $t$-structure on $\Coh^-(V \times X\B)$ with the property that the forgetful functor
	\[ \Coh^-(V \times X\B) \to \Coh^-(V \times X) \]
	is $t$-exact.
\end{prop}

\begin{proof}
	We know that
	\[ \Coh^-(V \times X\B) \simeq \Map_{\dAnSt}(V \times X\B, \bfAnCoh^-) \simeq \Map_{\Cat_\infty}(X\topl, \Coh^-(V)) . \]
	In particular, the $t$-structure on $\Coh^-(V)$ induces a $t$-structure on $\Coh^-(V \times X\B)$.
	We now claim that the forgetful functor
	\[ \Coh^-(V \times X\B) \to \Coh^-(V \times X) \]
	is $t$-exact.
	This property is local on $X$.
	We can therefore assume that $X\topl$ is contractible and that $X$ is Stein.
	Therefore $V \times X\B \simeq V$ and the canonical map $V \times X \to V \times X\B \simeq V$ can be identified with the projection.
	Since $X$ is smooth, this map is flat.
	This completes the proof.
\end{proof}

\begin{cor} \label{cor:De_Rham_Betti_t_exact}
	Let $X$ be a smooth analytic space.
	For any $V \in \dStn$ the canonical map
	\[ \Coh^-(V \times X\B) \to \Coh^-(V \times X\DR) \]
	is $t$-exact.
\end{cor}

\begin{proof}
	As a consequence of \cref{prop:t-structure_relative_crystals}, we see that it is enough to check that the composite
	\[ \Coh^-(V \times X\B) \to \Coh^-(V \times X\DR) \to \Coh^-(V \times X) \]
	is $t$-exact.
	This is guaranteed by \cref{prop:t_structure_Betti_form}.
\end{proof}

We can now take the first step toward the generalization of the Riemann-Hilbert correspondence we are concerned with in this paper:

\begin{prop} \label{prop:RH_almost_perfect_modules_ordinary_base}
	Let $X$ be a smooth analytic space and let $U \in \Stn$ be an ordinary Stein space.
	The map $\eta_{\mathrm{RH}} \colon X\DR \to X\B$ induces an equivalence of stable $\infty$-categories:
	\[ \eta_{\mathrm{RH}}^* \colon \Coh^-(U \times X\B) \to \Coh^-(U \times X\DR) . \]
\end{prop}

\begin{proof}
	Observe that both functors $\dAn_{\mathbb C}^0 \to \dAnSt$
	\[ X \mapsto X\DR , \qquad X \mapsto X\B \]
	take \v{C}ech nerves of coverings for the analytic topology $\tauan$ into colimit diagrams.
	Indeed, \cref{cor:De_Rham_colimits} implies the statement for the assignment $X \mapsto X\DR$.
	For what concerns the assignment $X \mapsto X\B$, it is enough to observe that the functor
	\[ (-)\topl \colon \dAn_{\mathbb C}^0 \to \cS \]
	takes indeed \v{C}ech nerves of coverings for the analytic topology into colimit diagrams.
	In particular, we can limit ourselves to prove the proposition in the special case where $X\topl$ is weakly contractible.
	
	\cref{cor:De_Rham_Betti_t_exact} guarantees that the map $\eta_{\mathrm{RH}}^*$ is $t$-exact.
	Using \cite[Lemma 7.13]{Hennion_Formal_gluing}, we see that it is enough to prove the following statements:
	\begin{enumerate}
		\item $\eta_{\mathrm{RH}}^*$ is fully faithful (in the derived sense) when restricted to $\Cohh(U \times X\B)$, and
		\item the induced map $\eta_{\mathrm{RH}}^* \colon \Cohh(U \times X\B) \to \Cohh(U \times X\DR)$ is essentially surjective.
	\end{enumerate}
	Since $U$ is an \emph{ordinary} Stein space, the second condition is satisfied in virtue of \cref{prop:comparison_RH} and of the classical Riemann-Hilbert correspondence of Deligne \cite[Théorème 2.23]{Deligne_Equations_differentielles}.
	We are therefore left to prove the fully faithfulness of $\eta_{\mathrm{RH}}^*$.
	
	We start by observing that if
	\[ F_0, F_1 \colon I \to \Cat_\infty \]
	are two functors and $f \colon F_0 \to F_1$ is a natural transformation with the property that for every $x \in I$ the induced functor
	\[ f_x \colon F_0(x) \to F_1(x) \]
	is fully faithful, then the induced functor
	\[ \overline{f} \colon \varprojlim_I F_0 \to \varprojlim_I F_1 \]
	is also fully faithful.
	Indeed, if $s, s' \in \varprojlim_I F_0$ are elements (thought as Cartesian sections of the Cartesian fibration associated to $F_0$), then we have
	\begin{align*}
		\Map_{\varprojlim_I F_0}(s,s') & \simeq \varprojlim_{x \in I} \Map_{F_0(x)}(s(x), s'(x)) \\
		& \simeq \varprojlim_{x \in I} \Map_{F_1(x)}(f_x(s(x)), f_x(s'(x))) \\
		& \simeq \Map_{\varprojlim_I F_1}(\overline{f}(s), \overline{f}(s')) .
	\end{align*}
	Combining this fact with the property of $X\B$ and $X\DR$ of commuting with geometric realizations of the \v{C}ech nerve of open covers, we can reduce ourselves to prove the statement in the special case where $X\topl$ is weakly contractible.
	\personal{Actually, we are also using the fact that in topoi colimits are universal.}
	Furthermore, combining this same argument with \cite[7.3.4.9]{HTT} we can replace $U$ with a Stein compact inside $U$.
	
	In this case, $X\B \simeq *$ and the Riemann-Hilbert natural transformation coincides with the projection
	\[ p \colon U \times X\DR \to U . \]
	We are therefore reduced to prove that $p^*$ is fully faithful when restricted to $\Cohh(U)$.
	
	Using \cref{lem:De_Rham_analytic} we see that $U \times X\DR \simeq | U \times (X^\bullet / X\DR) |$, and therefore we get an equivalence
	\[ p^{\bullet *} \colon \Coh^-(U \times X\DR) \to \varprojlim \Coh^-(U \times (X^\bullet / X\DR)) . \]
	Observe that we can promote these $\infty$-categories to symmetric monoidal $\infty$-categories and the functors $p^{n *}$ to symmetric monoidal functors.
	Moreover, these categories are closed.
	We claim that the functors $p^{n*}$ commute with the internal hom.
	Indeed, since we are assuming $U$ to be a Stein compact, we can use \cite[7.2.4.11]{Lurie_Higher_algebra} to reduce ourselves to show that
	\[ p^{n*} \cHom_U(\cF, \cG) \simeq \cHom_{U \times (X^n /X\DR)}(p^{n*} \cF, p^{n*} \cG)  \]
	in the special case where $\cF$ is a free sheaf.
	As the statement is tautological in this setting, we see that the claim is indeed proven.
	
	We now use the fact that both $U$ and $X$ are Stein spaces. In particular, if $\cF$ is a free sheaf and $\cG \in \Cohh(U)$, then
	\[ \rR^i \Gamma(U; \cHom_U(\cF, \cG)) \simeq \rR^i\Gamma(U \times (X^n / X\DR);  p^{n*} \cHom_{U \times (X^n / X\DR)}(\cF, \cG)) \simeq 0 . \]
	Moreover, the classical Riemann-Hilbert correspondence of \cite[Théorème 2.23]{Deligne_Equations_differentielles} implies that $p^*$ induces an equivalence
	\[ \pi_0 \Map_{\Cohh(U)}(\cF, \cG) \simeq \pi_0 \Map_{\Cohh(U \times X\DR)}(p^* \cF, p^*\cG) . \]
	Putting these information together, we conclude that the canonical map
	\[ \Map_{\Coh^-(U)}(\cF, \cG) \to \Map_{\Coh^-(U \times X\DR)}(p^* \cF, p^* \cG) \]
	is an equivalence whenever $\cF$ is free.
	Since $U$ is a Stein compact, we can now invoke \cite[7.2.4.11(5)]{Lurie_Higher_algebra} once again to deduce that the above map is an equivalence for every $\cF \in \Coh^-(U)$.
	This completes the proof.
\end{proof}

\begin{cor} \label{cor:RH_perfect_modules_ordinary_base}
	Let $X$ be a smooth analytic space and let $U \in \Stn$ be an ordinary Stein space.
	The map $\eta_{\mathrm{RH}} \colon X\DR \to X\B$ induces an equivalence of stable $\infty$-categories:
	\[ \eta_{\mathrm{RH}}^* \colon \Perf(U \times X\B) \to \Perf(U \times X\DR) . \]
\end{cor}

\begin{proof}
	We start by forming the commutative diagram
	\[ \begin{tikzcd}
		\Perf(U \times X\B) \arrow{r}{\eta_{\mathrm{RH}}^*} \arrow[hook]{d} & \Perf(U \times X\DR) \arrow[hook]{d} \\
		\Coh^-(U \times X\B) \arrow{r}{\eta_{\mathrm{RH}}^*} & \Coh^-(U \times X\DR) .
	\end{tikzcd} \]
	The vertical morphisms are fully faithful and the bottom horizontal map is an equivalence in virtue of \cref{prop:RH_almost_perfect_modules_ordinary_base}.
	It follows that the top horizontal map is fully faithful as well.
	
	We are left to prove that $\eta_{\mathrm{RH}}^* \colon \Perf(U \times X\B) \to \Perf(U \times X\DR)$ is essentially surjective.
	Since $\bfAnPerf$ is a sheaf, we can immediately reduce ourselves to the case where $X$ is a Stein space and $X\topl$ is weakly contractible, so that $U \times X\B \simeq U$.
	Choose $\cF \in \Perf(U \times X\DR)$.
	Reviewing $\cF$ as an element in $\Coh^-(U \times X\DR)$ and applying \cref{prop:RH_almost_perfect_modules_ordinary_base}, we can find $\cG \in \Coh^-(U)$ such that
	\[ \eta_{\mathrm{RH}}^*(\cG) \simeq \cF . \]
	It is then enough to prove that $\cG$ is a perfect complex on $U$.
	Choose a point $x \in X$.
	This point produces a section
	\[ s \colon U \simeq U \times *\DR \to U \times X\DR \]
	of the projection $p \colon U \times X\DR \to U$, which in turn can be identified with $\eta_{\mathrm{RH}}$.
	\personal{Here we are really using that $X\topl$ is contractible. Otherwise, we would not be able to get a map $U \times X\B \to U \times X\DR$!}
	In particular, we have
	\[ \cG \simeq s^* \eta_{\mathrm{RH}}^*(\cG) \simeq s^* (\cF) .  \]
	Since $\cF$ is a perfect complex, we conclude that the same goes for $\cG$.
	The proof is therefore complete.
\end{proof}

\section{Riemann-Hilbert correspondence with derived coefficients}

We need some categorical preliminary before tackling the proof of the main theorem.

\begin{defin}[{cf.\ \cite[4.7.5.16]{Lurie_Higher_algebra}}]
	We let $\Fun_{\mathrm{adj}}(\Delta^1, \Cat_\infty)$ be the $\infty$-subcategory of $\Fun(\Delta^1, \Cat_\infty)$ having as objects those morphisms $f \colon \cC \to \cD$ that admit a right adjoint, and having as morphisms the squares
	\[ \begin{tikzcd}
		\cC_1 \arrow{d}{f_1} \arrow{r}{h'} & \cC_2 \arrow{d}{f_2} \\ \cD_1 \arrow{r}{h} & \cD_2
	\end{tikzcd} \]
	that are right adjointable.
	That is, if $g_1$ and $g_2$ are the right adjoints for $f_1$ and $f_2$ respectively, we require the Beck-Chevally natural transformation
	\[ \alpha \colon h' g_1 \to g_2 h \]
	to be an equivalence.
\end{defin}

\begin{defin}
	We let $\Fun_{\mathrm{mon}}(\Delta^1, \Cat_\infty)$ be the full subcategory of $\Fun_{\mathrm{adj}}(\Delta^1, \Cat_\infty)$ spanned by those morphisms $f \colon \cC \to \cD$ having a \emph{monadic} right adjoint.
\end{defin}

\begin{lem} \label{lem:Beck_Chevalley_monadic_limits}
	\begin{enumerate}
		\item The $\infty$-category $\Fun_{\mathrm{adj}}(\Delta^1, \Cat_\infty)$ admits small limits and the inclusion
		\[ \Fun_{\mathrm{adj}}(\Delta^1, \Cat_\infty) \hookrightarrow \Fun(\Delta^1, \Cat_\infty) \]
		preserves them.
		\item The $\infty$-category $\Fun_{\mathrm{mon}}(\Delta^1, \Cat_\infty)$ is closed under limits in $\Fun_{\mathrm{adj}}(\Delta^1, \Cat_\infty)$.
		In particular, it admits small limits.
	\end{enumerate}
\end{lem}

\begin{rem}
	As M.\ Robalo pointed out to us, the first point of the above lemma already appeared as \cite[4.7.5.18]{Lurie_Higher_algebra}.
	For sake of clarity, we provide a proof, that is different from the one given in loc.\ cit.
\end{rem}

\begin{proof}
	Let $F \colon I \to \Fun_{\mathrm{adj}}(\Delta^1, \Cat_\infty)$ be a diagram.
	Associated to $F$ we have two diagrams
	\[ F_0, F_1 \colon I \to \Cat_\infty , \]
	plus a natural transformation
	\[ f \colon F_0 \to F_1 . \]
	Applying the unstraightening functor we thus obtain a morphism of Cartesian fibrations
	\[ \begin{tikzcd}
		\cX_0 \arrow{rr}{\varphi} \arrow{dr}[swap]{p_0} & & \cX_1 \arrow{dl}{p_1} \\ & I & \phantom{\cX_1} ,
	\end{tikzcd} \]
	where $p_i \colon \cX_i \to I$ is the Cartesian fibrations associated to $F_i$ (for $i = 0,1$) and $\varphi$ corresponds to $f$.
	The morphism $\varphi$ preserves Cartesian edges.
	We claim that $\varphi$ has a right adjoint relative to $I$ (in the sense of \cite[7.3.2.2]{Lurie_Higher_algebra}).
	In order to see this, it is enough to verify that the dual hypotheses of \cite[7.3.2.6]{Lurie_Higher_algebra} are satisfied.
	Condition (2) in loc.\ cit.\ is satisfied because $\varphi$ preserves Cartesian edges.
	As for condition (1), it is enough to observe that for every object $x \in I$, the morphism
	\[ f_x \colon F_0(x) \to F_1(x) \]
	has a right adjiont, by assumption.
	
	As a consequence, we obtain a functor of $\infty$-categories $\psi \colon \cX_1 \to \cX_0$, compatible with the projections to $I$.
	This functor is characterized by the fact that its fiber over $x \in I$ coincides with a right adjoint for $f_x$.
	We claim that $\psi$ preserves Cartesian edges.
	Indeed, after unraveling the definitions, this is equivalent to the statement that for every edge $\alpha \colon x \to y$ in $I$, the square
	\[ \begin{tikzcd}
		F_0(x) \arrow{r}{f_x} \arrow{d}{F_0(\alpha)} & F_1(x) \arrow{d}{F_1(\alpha)} \\
		F_0(y) \arrow{r}{f_y} & F_1(y)
	\end{tikzcd} \]
	is right adjointable, which is true by assumption.
	In particular, we see that $\psi$ gives rise to a natural transformation
	\[ g \colon F_0 \to F_1 , \]
	characterized by the property that $g_x$ is a right adjoint for $f_x$ for every $x \in I$.
	
	In particular, passing to the inverse limits, $f$ and $g$ induce a pair of functors
	\[ \overline{f} \colon \varprojlim_I F_0 \leftrightarrows \varprojlim_I F_1 \colon \overline{g} , \]
	that are compatible with the evaluation morphisms in the sense that the squares
	\[ \begin{tikzcd}
		\varprojlim_I F_0 \arrow[shift left = 1]{d}{\overline{f}} \arrow{r}{\ev_x^0} & F_0(x) \arrow[shift left = 1]{d}{f_x} \\
		\varprojlim_I F_1 \arrow{r}{\ev_x^1} \arrow[shift left = 1]{u}{\overline{g}} & F_1(x) \arrow[shift left = 1]{u}{g_x}
	\end{tikzcd} \]
	commute for every $x \in I$.
	The commutativity of the above diagram implies that $\overline{f}$ is a left adjoint for $\overline{g}$ and that the above square is right adjointable.
	In turn, this implies that
	\[ \overline{f} \colon \varprojlim_I F_0 \to \varprojlim_I F_1 \]
	is a limit diagram for $f \colon F_0 \to F_1$ in $\Fun_{\mathrm{adj}}(\Delta^1, \Cat_\infty)$.
	
	In order to prove statement (2), it is enough to prove that, with the above notations, if $g_x$ is monadic for every $x \in I$, then $\overline{g}$ is monadic as well.
	Invoking the Lurie-Barr-Beck theorem \cite[4.7.4.5]{Lurie_Higher_algebra}, we see that it is enough to check that if $V^\bullet$ is a $\overline{g}$-split simplicial object in $\varprojlim_I F_0$, then the colimit of $V^\bullet$ exists and it is preserved by $\overline{g}$.
	The very definition of $\overline{g}$-split simplicial object (see \cite[4.7.3.2]{Lurie_Higher_algebra}) implies that $\ev_x^0( V^\bullet )$ is a $g_x$-split simplicial object in $F_0(x)$.
	In particular, its colimit exists and it is preserved by $g_x$.
	Furthermore, since $g_x( \ev_x^0( V^\bullet ) ) \simeq \ev_x^1 ( \overline{g}( V^\bullet ) )$ is a split simplicial object in $F_1(x)$, we see that every functor defined on $F_1(x)$ commutes with the geometric realization of this simplicial object.
	This has the following consequence: if $\alpha \colon x \to y$ is a morphism in $I$, then $F_0(\alpha)$ commutes with the geometric realization of $\ev_x^0( V^\bullet )$.
	Indeed, since $g_y$ is monadic and $F_0(\alpha)( \ev_x^0 (V^\bullet))$ is $g_y$-split, it is enough to prove that $g_y \circ F_0(\alpha) \simeq F_1(\alpha) \circ g_x$ commutes with such geometric realization, which is true in virtue of the above consideration.
	This shows that the colimit of $V^\bullet$ exists after applying each evaluation $\ev^0_x$, and that such colimit is preserved by the transition morphisms of the diagram $F_0$.
	It follows that the colimit of $V^\bullet$ exists in $\varprojlim_I F_0$ and that it is preserved by $\overline{g}$.
	In conclusion, $\overline{g}$ is monadic and so the proof is complete.
\end{proof}

We use the above categorical considerations in the following way.
Let $U \in \dStn$ and let $\cF \in \Coh^{\le -1}(V)$.
Consider the split square-zero extension $U[\cF]$ together with the natural map
\[ \pi \colon U[\cF] \to U , \]
corresponding to the zero derivation with coefficients in $\cF$.

Let $V \in \dStn$ be any other space and consider the diagram
\[ \begin{tikzcd}
	V \times U[\cF] \arrow{r}{\pi_V} \arrow{d}{q} & V \times U \arrow{d}{p} \arrow{r} & V \arrow{d} \\
	U[\cF] \arrow{r}{\pi} & U \arrow{r} & *_{\mathbb C}
\end{tikzcd} \]
The square on the right and the outer one are derived pullbacks by definition.
It follows that the one on the left is a derived pullback as well.

Notice furthermore that since $\rH^0(\cF) = 0$, the map $\pi_0$ is a closed immersion.
In particular, the induced map $V \times U[\cF] \to V \times U$ is a closed immersion as well.
The derived proper base change implies:

\begin{lem}
	The diagram
	\[ \begin{tikzcd}
		\Coh^-(V \times U[\cF]) & \Coh^-(V \times U) \arrow{l}[swap]{\pi_V^*} \\
		\Coh^-(U[\cF]) \arrow{u}[swap]{q^*}  & \Coh^-(U) \arrow{u}{p^*} \arrow{l}[swap]{\pi^*}
	\end{tikzcd} \]
	is right adjointable.
\end{lem}

\begin{proof}
	This is a special case of \cref{thm:unramifiedness}.
\end{proof}

As consequence, we see that the map
\[ \pi^* \colon \bfAnCoh^-(U) \to \bfAnCoh^-(U[\cF]) , \]
seen as a functor
\[ A_\pi \colon \dStn\op \to \Fun(\Delta^1, \Catst) , \]
actually factors through
\[ A_\pi \colon \dStn\op \to \Fun_{\mathrm{adj}}(\Delta^1, \Catst) . \]
Moreover, since for every $V \in \dStn$ the map $\pi_{V_*} \colon \Coh^-(V \times U[\cF]) \to \Coh^-(V \times U)$ is monadic, we see that $A_\pi$ actually factors through $\Fun_{\mathrm{mon}}(\Delta^1, \Catst)$.
Since \cref{lem:Beck_Chevalley_monadic_limits} guarantees that limits in $\Fun_{\mathrm{adj}}(\Delta^1, \Catst)$ are computed objectwise, we see that $A_\pi$ is a sheaf for the $\tauan$-topology.
In particular, we obtain the following key result:

\begin{cor} \label{cor:monadic_right_adjointable_square}
	Let $f \colon F \to G$ be a morphism in $\dAnSt$.
	Then the induced diagram
	\[ \begin{tikzcd}
		\Map_{\dAnSt^{\mathrm{cat}}}(G, \bfAnCoh^-(U)) \arrow{r}{\pi_G^*} \arrow{d}{f^*} & \Map_{\dAnSt^{\mathrm{cat}}}(G, \bfAnCoh^-(U[\cF])) \arrow{d}{f^*} \\
		\Map_{\dAnSt^{\mathrm{cat}}}(F, \bfAnCoh^-(U))) \arrow{r}{\pi_F^*} & \Map_{\dAnSt^{\mathrm{cat}}}(F, \bfAnCoh^-(U[\cF]))
	\end{tikzcd} \]
	is right adjointable and the right adjoints $\pi_{G*}$ and $\pi_{F*}$ are monadic.
\end{cor}

This allows us to handle the key inductive step:

\begin{cor} \label{cor:RH_inductive_step}
	Let $X$ be a smooth analytic space and suppose that $X\topl$ is weakly contractible.
	Let $U \in \dStn$ be such that the application $\eta_{\mathrm{RH}} \colon X\DR\to X\B$ induces an equivalence
	\[ \Coh^-(X\B \times U) \to \Coh^-(X\DR \times U) . \]
	Then for any $M \in \Coh^{\le -1}(U)$, the application $\eta_{\mathrm{RH}}$ induces an equivalence
	\[ \Coh^-(X\B \times U[M]) \to \Coh^-(X\DR \times U[M]) . \]
\end{cor}

\begin{proof}
	Applying \cref{cor:monadic_right_adjointable_square} to the map $\eta_{\mathrm{RH}} \colon X\DR \to X\B$, we obtain the following right adjointable square:
	\[ \begin{tikzcd}
			\Coh^-(X\B \times U) \arrow{r}{\pi\B^*} \arrow{d}{\eta_{\mathrm{RH}}^*} & \Coh^-(X\B \times U[\cF]) \arrow{d}{\eta_{\mathrm{RH}}^*} \\
			\Coh^-(X\DR \times U) \arrow{r}{\pi\DR^*} & \Coh^-(X\DR \times U[\cF]) .
	\end{tikzcd} \]
	The hypothesis guarantees that the left vertical map is an equivalence.
	Furthermore, the right adjoints $\pi_{\rB *}$ and $\pi_{\mathrm{dR}*}$ are monadic.
	It is now sufficient to invoke \cite[4.7.4.16]{Lurie_Higher_algebra} to conclude that the right vertical map is an equivalence as well.
\end{proof}

We only need one extra step in order to be ready for the proof of the derived Riemann-Hilbert correspondence.
We start with a definition:

\begin{defin} \label{def:inf_cart_convergence}
	Let $F \in \dAnSt^{\mathrm{cat}}$.
	We say that:
	\begin{enumerate}
		\item $F$ is \emph{strongly infinitesimally cartesian} if for every $V \in \dStn$, every $M \in \Coh^{\le -1}(V)$ and every analytic derivation $d \colon V[M] \to V$, the natural morphism
		\[ \bfMap(V_d[M[1]], F) \to \bfMap(V, F) \times_{\bfMap(V[M], F)} \bfMap(V, F) \]
		is an equivalence in $\dAnSt$.
		\item $F$ is \emph{strongly convergent} if for every $V \in \dStn$ the natural morphism
		\[ \bfMap(V, F) \to \varprojlim_n \bfMap (\mathrm t_{\le n}(V), F) \]
		is an equivalence.
	\end{enumerate}
\end{defin}

\begin{rem} \label{rem:weak_inf_cart_implies_strong_inf_cart}
	By replacing the mapping stacks with the mapping spaces in $\dAnSt$ in the above definition we get the versions of being infinitesimally cartesian and convergent.
	Clearly, if a derived analytic stack $F$ is strongly infinitesimally cartesian (resp.\ strongly convergent), then $F$ is also infinitesimally cartesian (resp.\ convergent).
	It turns out that the vice-versa holds as well.
	This is easy to see in the case of the infinitesimally cartesian property.
	Indeed, suppose that $F$ is infinitesimally cartesian and let $V$, $M$ and $d \colon V[M] \to V$ as in the above definition.
	We have to prove that for every $U \in \dStn$, the map
	\begin{equation} \label{eq:strong_inf_cart}
		\Map(U, \bfMap(V_d[M[1]]), F) \to \Map(U,\bfMap(V, F)) \times_{\Map(U, \bfMap(V[M], F))} \Map(U, \bfMap(V, F))
	\end{equation}
	is an equivalence.
	Using the universal property of the mapping stack, we can rewrite the above map as
	\[ \Map(U \times V_d[M[1]], F) \to \Map(U \times V, F) \times_{\Map(U \times V[M], F)} \Map(U \times V, F) . \]
	If $q \colon U \times V \to V$ denotes the canonical projection, then we remark that
	\[ U \times (V[M]) \simeq (U \times V)[q^* M] , \]
	and similarly
	\[ U \times (V_d[M[1]]) \simeq (U \times V)_{q^*(d)}(q^*M[1]) . \]
	Therefore, it is enough to apply the infinitesimally cartesian property of $F$ to $U \times V$, $q^*(M)$ and to the analytic derivation $q^*(d)$ to conclude that the map \eqref{eq:strong_inf_cart} is an equivalence as well.
	
	For the convergency property, there is a similar argument, which is however slightly trickier.
	In the algebraic setting, we refer to \cite[Proposition 2.1.8]{DAG-XIV}.
	The same argument applies in the analytic setting thanks to the results of \cite{Porta_DCAGII}.
\end{rem}

\begin{lem} \label{lem:infinitesimally_cartesian}
	\begin{enumerate}
		\item The stacks $\bfAnCoh^-$ and $\bfAnPerf$ are strongly infinitesimally cartesian and strongly convergent;
		\item for any derived analytic stack in $\infty$-categories $G \in \dAnSt^{\mathrm{cat}}$, the stacks
		\[ \bfAnCoh^-(G) \coloneqq \bfMap(G, \bfAnCoh^-) \quad \text{and} \quad \bfAnPerf(G) \coloneqq \bfMap(G, \bfAnPerf) \]
		are strongly infinitesimally cartesian and strongly convergent.
	\end{enumerate}
\end{lem}

\begin{proof}
	The second statement follows from the first as follows: let $V \in \dStn$ and let $M \in \Coh^{\le -1}(V)$.
	Let $s \colon V[M] \to V$ be any analytic derivation and consider the associated analytic square-zero extension $V_s[M[1]]$.
	Then condition (1) implies that
	\[ \bfAnCoh^-(V_s[M[1]]) \simeq \bfAnCoh^-(V) \times_{\bfAnCoh^-(V[M])} \bfAnCoh^-(V) , \]
	in $\dAnSt$.
	It follows that
	\[ \bfMap(G, \bfAnCoh^-(V_s[M[1]])) \simeq \bfMap(G, \bfAnCoh^-(V)) \times_{\bfMap(G, \bfAnCoh^-(V[M]))} \bfMap(G, \bfAnCoh^-(V)) \]
	is an equivalence.
	Using the universal property of the mapping stack, we can rewrite the above equivalence as
	\[ \bfMap(V_s[M[1]], \bfAnCoh^-(G)) \simeq \bfMap(V, \bfAnCoh^-(G)) \times_{\bfMap(V[M], \bfAnCoh^-(G))} \bfMap(V, \bfAnCoh^-(G)) . \]
	The proof for the convergency property is analogous.
	
	We are therefore left to prove that $\bfAnCoh^-$ (and $\bfAnPerf$) are strongly infinitesimally cartesian and strongly convergent.
	Moreover, in virtue of \cref{rem:weak_inf_cart_implies_strong_inf_cart}, we can drop the adjective ``strong''.
	We start by convergency.
	Let $V \in \dStn$ and represent it as a structured $\cTan$-topos $(\cV, \cO_V)$.
	Then
	\[ \mathrm t_{\le n}(V) \simeq (\cV, \tau_{\le n} \cO_V) . \]
	We now have that
	\[ \cO_V \Mod \simeq \varprojlim \tau_{\le n} \cO_V \Mod , \]
	and this equivalence respects (almost) perfect modules.
	In particular, we see that $\bfAnCoh^-$ and $\bfAnPerf$ are convergent.
	
	We now turn to the infinitesimally cartesian property.
	Let $V \in \dStn$, $M \in \Coh^{\le - 1}(V)$ and $s \colon V[M] \to V$ be as in \cref{def:inf_cart_convergence}.
	Represent again $V$ as a $\cTan$-structured topos $V = (\cV, \cO_V)$.
	Then
	\[ V[M] \simeq (\cV, \cO_V \oplus M) , \]
	where $\cO_V \oplus M$ is the analytic split square-zero extension of $\cO_V$ by $M$.
	Similarly,
	\[ V_s[M[1]] = (\cV, \cO_V \oplus_s M[1]) , \]
	where $\cO_V \oplus_s M[1]$ is the analytic square-zero extension of $\cO_V$ by $M$ determined by $s$.
	Since $\bfAnCoh^-$ and $\bfAnPerf$ are sheaves for the analytic topology, it is enough to reason locally.
	Using \cite[7.3.4.9]{HTT}, we can reduce ourselves to the case of a Stein compact in $V$.
	In this case, the lemma is a direct consequence of \cite[Proposition 7.6]{DAG-IX}.
\end{proof}

We are finally ready to prove the main result of this paper:

\begin{thm} \label{thm:derived_RH}
	Let $X$ be a smooth analytic space.
	The Riemann-Hilbert map $\eta_{\mathrm{RH}} \colon X\DR \to X\B$ induces an equivalences
	\[ \eta_{\mathrm{RH}}^* \colon \bfAnCoh^-(X\B) \to \bfAnCoh^-(X\DR) \]
	and
	\[ \eta_{\mathrm{RH}}^* \colon \bfAnPerf(X\B) \to \bfAnPerf(X\DR) \]
\end{thm}

\begin{proof}
	Since both $\bfAnCoh^-(X\B)$ and $\bfAnCoh^-(X\DR)$ are sheaves in $X$, it is enough to prove the assertion when $X\topl$ is contractible.
	In this way, $X\B \simeq *$ and $\eta_{\mathrm{RH}}$ can be identified with the canonical morphism $X\DR \to *$.
	In order to prove the theorem, we have to check that for every $V \in \dStn$, the canonical projection $p \colon V \times X\DR \to V$ induces an equivalence
	\[ p^* \colon \Coh^-(V) \to \Coh^-(V \times X\DR) . \]
	
	We proceed by induction on the Postnikov tower of $V$ (note that this is made possible by the results of \cite{Porta_DCAGII}).
	When $V$ is discrete, the desired result coincides exactly with the statements of \cref{prop:RH_almost_perfect_modules_ordinary_base} and of \cref{cor:RH_perfect_modules_ordinary_base}.
	Suppose now that the statement holds true for $\mathrm t_{\le n}(V)$.
	Set $M \coloneqq \pi_{n+1}(\cO_V)[n+2]$.
	Then there exists an analytic derivation $s \colon \mathrm t_{\le n}(V)[M] \to \mathrm t_{\le n}(V)$ such that the diagram
	\[ \begin{tikzcd}
		\mathrm t_{\le n}(V)[M] \arrow{r}{s} \arrow{d}{s_0} & \mathrm t_{\le n}(V) \arrow{d} \\
		\mathrm t_{\le n}(V) \arrow{r} & \mathrm t_{\le n + 1}(V)
	\end{tikzcd} \]
	is a pushout in $\dAn_{\mathbb C}$.
	Since $\bfAnCoh^-(X\B)$ and $\bfAnCoh^-(X\DR)$ are infinitesimally convergent by \cref{lem:infinitesimally_cartesian}, they both take this pushout to a pullback square.
	Since by induction hypothesis we have an equivalence
	\[ \Coh^-(\mathrm t_{\le n}(V)) \to \Coh^-(\mathrm t_{\le n}(V) \times X\DR) , \]
	we are then reduced to prove that the map
	\[ \Coh^-(\mathrm t_{\le n}(V)[M]) \to \Coh^-(\mathrm t_{\le n}(V)[M] \times X\DR) \]
	is an equivalence as well.
	This is precisely what \cref{cor:RH_inductive_step} guarantees to be true.
\end{proof}

\begin{rem}
	The equivalence proven in \cref{thm:derived_RH} can be promoted to an equivalence of stacks with values in symmetric monoidal stable $\infty$-categories.
	Indeed, both $\bfAnCoh^-$ and $\bfAnPerf$ can be promoted to functors
	\[ \bfAnCoh^{-, \otimes}, \bfAnPerf^\otimes \colon \dStn\op \to \Cat_\infty^{\mathrm{st}, \otimes} , \]
	and the forgetful functor $\Cat_\infty^{\mathrm{st}, \otimes} \to \Cat_\infty$ is conservative.
\end{rem}

\begin{rem}
	The version of the Riemann-Hilbert correspondence can be used to prove the following (weaker) statement: let $X$ be a smooth and proper algebraic variety.
	Then there is an equivalence
	\[ \bfMap(X\B, \bfPerf)\an \simeq \bfMap(X\DR, \bfPerf)\an . \]
	Indeed, this equivalence is a consequence of \cref{thm:derived_RH} and the following facts:
	\begin{enumerate}
		\item the analytification of $\bfAnPerf$ coincides with the analytic stack of perfect complexes $\bfAnPerf$;
		\item there are equivalences
		\[ \bfMap(X\B, \bfPerf)\an \simeq \bfMap((X\an)\B, \bfAnPerf) \simeq \bfMap(X\DR, \bfPerf)\an \simeq \bfMap((X\an)\DR, \bfAnPerf) . \]
	\end{enumerate}
	We expect both these properties to be true.
	Proofs will appear in a separate paper, in collaboration with C.\ Di Natale and J.\ Holstein \cite{DiNatale_Analytification_mapping_stacks}.
\end{rem}

\bibliographystyle{plain}
\bibliography{dahema}

\end{document}